\documentclass[a4paper,reqno]{amsart}

\usepackage[british]{babel}

\usepackage{mlmodern}
\DeclareFontFamily{OMX}{mlmex}{}
\DeclareFontShape{OMX}{mlmex}{m}{n}{<->mlmex10}{}
\usepackage[T1]{fontenc}
\usepackage[babel,stretch=10,shrink=10,verbose=errors,selected]{microtype}
\usepackage[strict,english=british]{csquotes}

\usepackage{amsmath,amssymb,amsthm}
\usepackage{mathtools}
\usepackage[all]{onlyamsmath}

\usepackage[
backend=biber,
style=numeric-comp,
sorting=nyt,
sortcites=true,
maxnames=5,
giveninits=true,
date=year,
useprefix=true,
]{biblatex}
\DeclareFieldFormat{eid}{no.~#1}
\usepackage{enumitem}

\usepackage{tikz}
\usetikzlibrary{cd}

\usepackage{xcolor}
\usepackage[nobiblatex]{xurl}
\usepackage[breaklinks=true,pdflang=en-GB,colorlinks=true]{hyperref}
\colorlet{citecolor}{green!75!black}
\colorlet{linkcolor}{red!75!black}
\colorlet{urlcolor}{blue!75!black}
\hypersetup{citecolor=citecolor,linkcolor=linkcolor,urlcolor=urlcolor}
\usepackage{zref-clever}
\zcsetup{cap,nameinlink=false}
\usepackage{keytheorems}

\usepackage{orcidlink}
\usepackage{ellipsis}

\numberwithin{equation}{section}

\newkeytheorem{theorem}[parent=section,style=plain]
\newkeytheorem{proposition}[sibling=theorem,style=plain]
\newkeytheorem{lemma}[sibling=theorem,style=plain]
\newkeytheorem{corollary}[sibling=theorem,style=plain]

\newkeytheorem{definition}[sibling=theorem,style=definition]
\newkeytheorem{example}[sibling=theorem,style=definition]

\newkeytheorem{remark}[sibling=theorem,style=remark]

\newcommand{\set}[2]{\ensuremath{\{\,{#1}\mid{#2}\,\}}}
\newcommand{\card}[1]{\ensuremath{\left|{#1}\right|}}
\newcommand{\dcs}[3]{\ensuremath{{#1}\backslash{#2}/{#3}}}
\newcommand{\aff}[3]{\ensuremath{{#1}\backslash_{#2}{#3}}}
\newcommand{\id}{\ensuremath{\mathrm{id}}}

\newcommand{\NN}{\mathbb{N}}
\newcommand{\ZZ}{\mathbb{Z}}

\DeclareMathOperator{\im}{im}
\DeclareMathOperator{\GL}{GL}

\DeclareMathOperator{\Aff}{Aff}
\DeclareMathOperator{\Aut}{Aut}
\DeclareMathOperator{\Hom}{Hom}
\DeclareMathOperator{\Coin}{Coin}

\newcommand{\normalsub}{\trianglelefteq}

\newcommand{\ind}[2]{\ensuremath{[#1\mathbin{:}#2]}}
\newcommand{\R}{\mathcal{R}}

\title{A Hirsch length inequality}
\author[S. Tertooy]{Sam Tertooy\ \orcidlink{0000-0002-5750-9153}}
\date{\today}
\address{KU Leuven, Kulak Kortrijk Campus\\
	E.~Sabbelaan 53\\
	8500 Kortrijk\\
	Belgium}
\email{\href{mailto:sam.tertooy@kuleuven.be}{sam.tertooy@kuleuven.be}}
\urladdr{\url{https://stertooy.github.io}}


\begin{document}
	
	\begin{abstract}
		Let \(H\) and \(K\) be subgroups of a virtually polycyclic group \(G\). We prove the Hirsch length inequality
		\[
		h(H)+h(K) \leq h(H\cap K)+h(G).
		\]
		We show that equality holds when the number of \((H,K)\)-double cosets is finite, and that the converse holds when \(G\) is nilpotent. For homomorphisms \(\varphi,\psi \colon G \to H\) between virtually polycyclic groups, the equivalent formulation in terms of their coincidence subgroup is
		\[
		h(\Coin(\varphi,\psi)) \geq h(G)-h(H).
		\]
		In this formulation, equality holds when the Reidemeister number \(R(\varphi,\psi)\) is finite, and the converse holds when \(H\) is nilpotent.
	\end{abstract}
	
	\maketitle

	\section{Introduction}
	
	For finite groups, there is a well-known formula linking the sizes of two subgroups \(H\) and \(K\), their intersection \(H \cap K\), and the double coset \(HK\):
	\[ \card{H}\card{K} = \card{H \cap K}\card{HK}.\]
	For infinite (sub)groups, though, this equality just ends up stating that ``\(\infty = \infty\)''. We must then turn our attention to other ways of measuring the size of a subgroup. For example, in the family of virtually polycyclic groups, there is the Hirsch length. This is defined as the number of infinite factors in a subnormal series whose factors are either finite or infinite cyclic.
	
	If \(H\) and \(K\) are subgroups of a virtually polycyclic group \(G\) with \(H \leq N_G(K)\), such that \(HK\) is also a subgroup of \(G\), then the second isomorphism theorem readily produces an analogue of the above formula in terms of the Hirsch lengths:
	\[ h(H) + h(K) = h(H \cap K) + h(HK).\]
	But what if \(H\) is not a subgroup of \(N_G(K)\)? Then the double coset \(HK\) is (in general) not a subgroup of \(G\), so the above equality is ill-defined. We show that by replacing \(h(HK)\) with \(h(G)\), we can still produce an inequality, and that equality is closely related to the finiteness of the number of \((H,K)\)-double cosets \(\card{\dcs{H}{G}{K}}\).
	
	\begin{theorem}[manual-num=A,store=mainA]
		\label{thm:mainresultA}
		Let \(G\) be a virtually polycyclic group with subgroups \(H, K\). Then
		\[ h(H) + h(K) \leq h(H \cap K) + h(G),\]
		and moreover:
		\begin{itemize}
			\item if \(\card{\dcs{H}{G}{K}}\) is finite, then the above is an equality;
			\item if \(G\) is nilpotent and the above is an equality, then \(\card{\dcs{H}{G}{K}}\) is finite.
		\end{itemize}
	\end{theorem}
	
	The inequality in the above \zcref[noref,nocap]{thm:mainresultA} was stated in \cite[Prop.~3]{wong00-a}; however, the proof given there contains a gap which we discuss in \zcref{sec:twicon}. We provide a new proof of the inequality, and we provide implications for the finiteness of the number of double cosets. In the specific case of \(G\) being a finitely generated torsion-free nilpotent group, \zcref{thm:mainresultA} follows immediately from \cite[Thm.~4.11]{hly23-a}. 
	
	We also give an equivalent formulation of \zcref{thm:mainresultA} in terms of coincidence and twisted conjugacy. For two homomorphisms \(\varphi,\psi \colon G \to H\), their coincidence subgroup is defined as \(\Coin(\varphi,\psi) \coloneq \set{g \in G}{\varphi(g) = \psi(g)}\), and their Reidemeister number \(R(\varphi,\psi)\) is the number of orbits of the action of \(G\) on \(H\) given by \(g \cdot h \coloneq \varphi(g)h\psi(g)^{-1}\).
	
	\begin{theorem}[manual-num=B,store=mainB]
		\label{thm:mainresultB}
		Let \(G, H\) be virtually polycyclic groups and \(\varphi,\psi \in \Hom(G,H)\). Then
		\[ h(\Coin(\varphi,\psi)) \geq h(G) - h(H),\]
		and moreover:
		\begin{itemize}
			\item if \(R(\varphi,\psi)\) is finite, then the above is an equality;
			\item if \(H\) is nilpotent and the above is an equality, then \(R(\varphi,\psi)\) is finite.
		\end{itemize}
	\end{theorem}
	
	This formulation combines and extends results for torsion-free polycyclic groups from \cite{wong00-a}. 
	
	This paper is structured as follows. In \zcref{sec:prelim} we recall the required definitions and properties of the Hirsch length, group derivations, and rationally irreducible modules. In \zcref{sec:abderiv} we analyse derivations between abelian groups, and in \zcref{sec:setup} we introduce a setup for the subsequent proofs that lets us leverage these derivations. In \zcref{sec:basecase} we prove \zcref{thm:mainresultA} for metabelian groups; in \zcref{sec:maincase} we prove it in full generality. Finally, in \zcref{sec:twicon} we apply our main result to twisted conjugacy.
	
	\section{Preliminaries}
	\label{sec:prelim}
	
	\subsection{Virtually polycyclic groups and Hirsch length}
	
	We recall the definitions of a (virtually) polycyclic group and of the Hirsch length, and some of their properties, following \cite[Ch.~1]{sega83-a}.
	
	\begin{definition}
		A group \(G\) is called \emph{polycyclic} if it admits a finite subnormal series
		\[ 1 = G_0 \normalsub G_1 \normalsub \cdots \normalsub G_{k-1} \normalsub G_k = G\]
		such that every factor \(G_i / G_{i-1}\) is cyclic. A group is called \emph{virtually polycyclic} if it admits a finite-index polycyclic subgroup.
	\end{definition}
	
	\begin{proposition}
		A virtually polycyclic group \(G\) admits a finite subnormal series
		\[ 1 = G_0 \normalsub G_1 \normalsub \cdots \normalsub G_{k-1} \normalsub G_k = G\]
		such that every factor \(G_i / G_{i-1}\) is either infinite cyclic or finite.
	\end{proposition}
	
	\begin{definition}
		The \emph{Hirsch length} \(h(G)\) of a virtually polycyclic group \(G\) is the number of infinite cyclic factors in a subnormal series with infinite cyclic or finite factors.
	\end{definition}
	
	We remark that the Hirsch length is independent of the chosen subnormal series.
	
	\begin{proposition}
		\label{prop:hirschprops}
		Let \(G\) be a virtually polycyclic group with a subgroup \(H\) and a normal subgroup \(N\). Then
		\begin{enumerate}
			\item \(h(H) \leq h(G)\), with equality if and only if \(\ind{G}{H} < \infty\),
			\item \(h(G) = h(N) + h(G/N)\).
		\end{enumerate}
	\end{proposition}
	
	\subsection{Group derivations}
	
	If a group \(G\) acts on another group \(H\) in such a way that
	\[ g \cdot (h_1h_2) = (g \cdot h_1)(g \cdot h_2),\]
	we say that \(G\) acts on \(H\) \emph{via automorphisms}. This naming convention makes sense: for any such action, there is a homomorphism \(\lambda \colon G \to \Aut(H)\) such that \(g \cdot h = \lambda(g)(h)\) for all \(g \in G\) and all \(h \in H\). If \(G\) acts via automorphisms on an abelian group \(M\), then \(M\) is called a \emph{\(G\)-module}. 
	
	\begin{definition}
		Let \(G, H\) be groups and let \(G\) act on \(H\) via automorphisms. A \emph{(group) derivation} is a map
		\(\delta\colon G \to H\) such that
		\begin{equation}
			\label{eqn:derivationidentity}
			\delta(g_1g_2) = \delta(g_1)(g_1\cdot \delta(g_2))
		\end{equation}
		for all \(g_1,g_2 \in G\).
	\end{definition}
	
	We refer to \cite[Sec.~IV.2]{brow82-a} for a standard treatment of derivations with abelian codomain. The identities in \zcref{prop:derivprops} below also hold for non-abelian codomains.
	
	The notion of group derivation generalises that of group homomorphism. A derivation \(\delta \colon G \to H\) is a homomorphism when \(G\) acts trivially on its image. In particular, this happens when \(G\) acts trivially on \(H\).
	
	\begin{definition}
		Let \(G,H\) be groups, let \(G\) act on \(H\) via automorphisms and let \(\delta \colon G \to H\) be a derivation. We define the \emph{affine action} of \(G\) on \(H\) via \(\delta\) by
		\[ g \ast_\delta h \coloneq \delta(g)(g \cdot h).\]
		We use \(G \ast_\delta h\) to denote the orbit of \(h \in H\), and \(\aff{G}{\delta}{H}\) to denote the orbit space.
	\end{definition}
	Affine actions can be seen as a generalisation of actions via automorphisms. If the action of \(G\) on \(H\) via automorphisms corresponds to the homomorphism \(\lambda \colon G \to \Aut(H)\), and \(\delta \colon G \to H\) is a derivation, then the affine action via \(\delta\) corresponds to the homomorphism
	\[ G \to \Aff(H) = H \rtimes \Aut(H) \colon g \mapsto (\delta(g),\lambda(g)).\]
	The following \zcref[nocap,noref]{prop:derivprops} lists some properties of group derivations. We include a short proof, to illustrate that we do not need commutativity of the codomain.
	
	\begin{proposition}
		\label{prop:derivprops}
		Let \(\delta \colon G \to H\) be a group derivation and let \(g, g_1, g_2 \in G\). Then
		\begin{enumerate}[label=(\arabic*)]
			\item \(\delta(1_G) = 1_H\),
			\item \(\delta(g^{-1}) = g^{-1} \cdot \delta(g)^{-1}\),
			\item \(\delta(g_1) = \delta(g_2)\) if and only if \(\delta(g_1^{-1}g_2) = 1_H\),
			\item \(\ker \delta\) is a subgroup of \(G\).
		\end{enumerate}
	\end{proposition}
	\begin{proof}
		The first identity follows from taking \(g_1 = g_2 = 1_G\) in \zcref[noname]{eqn:derivationidentity}. We then use this to prove (2):
		\[
		1_H = \delta(1_G) = \delta(g^{-1} g)
		= \delta(g^{-1}) \left( g^{-1} \cdot \delta(g)\right)
		= \delta(g^{-1}) \left( g^{-1} \cdot \delta(g)^{-1}\right)^{-1}.
		\]
		Consequently, 
		\[
		\delta(g_1^{-1}g_2)
		= \delta(g_1^{-1}) \left( g_1^{-1} \cdot \delta(g_2) \right)
		= g_1^{-1} \cdot \left( \delta(g_1)^{-1}\delta(g_2) \right),
		\]
		which proves (3). Moreover, the above also shows that if \(g_1, g_2 \in \ker \delta\), then \(g_1^{-1}g_2 \in \ker \delta\), proving (4).
	\end{proof}
	
	When comparing (4) with group homomorphisms, notice that we did not state that \(\ker \delta\) is a \emph{normal} subgroup, nor did we state that \(\im \delta\) is a subgroup of \(H\). We illustrate with two examples that neither assertion holds.
	\begin{example}
		\label{exm:derivimgnonsubgrp}
		Let \(\ZZ_2\) act on \(\ZZ_3\) by inversion, which is an action via automorphisms since we are dealing with abelian groups. Let \(x\) and \(y\) be generators of \(\ZZ_2\) and \(\ZZ_3\), respectively. The group derivation \(\delta \colon \ZZ_2 \to \ZZ_3\) defined by
		\[\delta(1) \coloneq 1, \qquad \delta(x) \coloneq y\]
		has image \(\im \delta = \{1,y\}\), which is not a subgroup of \(\ZZ_3\).
	\end{example}
	
	\begin{example}
		\label{exm:derivkernonnormal}
		Let \(S_3\) act on \(\ZZ_2 \times \ZZ_2\) by permuting its non-identity elements. Fix some \(x \neq 1\) in \(\ZZ_2 \times \ZZ_2\). Then the map
		\[ \delta \colon S_3 \to \ZZ_2 \times \ZZ_2 \colon g \mapsto x(g \cdot x)\]
		is a derivation. Its kernel is then the subgroup of permutations that fix \(x\), which is isomorphic to \(\ZZ_2\), and the subgroups of order \(2\) in \(S_3\) are never normal.
	\end{example}
	\subsection{Rationally irreducible modules} To finish this section, we give a short introduction to rationally irreducible modules and pure subgroups, which will be necessary in the next section. A more elaborate treatment can be found in \cite[Ch.~2]{sega83-a} and \cite[Sec.~1.4]{dks17-a}.
	
	\begin{definition}
		A \(G\)-module \(M \cong \ZZ^n\), with \(n \geq 1\), is called \emph{rationally irreducible} if every non-trivial \(G\)-invariant subgroup has finite index in \(M\).
	\end{definition}
	
	\begin{definition}
		A subgroup \(B\) of an abelian group \(A\) is called \emph{pure} if for every element \(a \in A\) and every integer \(n \geq 1\) such that \(a^n \in B\), there exists an element \(b \in B\) such that \(b^n = a^n\).
	\end{definition}
	
	If moreover \(A\) is torsion-free, then the condition above can be restated as ``if \(a^n \in B\) for some \(a \in A\) and \(n \geq 1\), then \(a \in B\)''.
	
	\begin{proposition}
		A \(G\)-module \(M \cong \ZZ^n\), with \(n \geq 1\), is rationally irreducible if and only if its only \(G\)-invariant pure subgroups are the trivial subgroup and \(M\) itself.
	\end{proposition}
	
	The following \zcref[noref,nocap]{thm:irredhirsch} is a direct consequence of \cite[Thm.~2.1]{frei71-a}.
	\begin{theorem}
		\label{thm:irredhirsch}
		Let \(G\) be an abelian subgroup of \(\GL(n,\ZZ)\), with \(n \geq 1\). If \(\ZZ^n\) is rationally irreducible as a \(G\)-module, then \(G\) is finitely generated and \(h(G) \leq n-1\).
	\end{theorem}
	
	\section{Derivations between abelian groups}
	\label{sec:abderiv}
	
	The main result of this section is \zcref{prop:abelianderv}, which proves an inequality involving the Hirsch lengths of the domain, codomain and kernel of a derivation between finitely generated abelian groups. Moreover, it shows that equality is attained if the associated affine action has only finitely many orbits. This result will be used to prove the metabelian case in \zcref{lem:metabeliancase}; after proving \zcref{thm:mainresultA}, we obtain a generalisation to derivations between arbitrary virtually polycyclic groups in \zcref{cor:derivationmainresult}.
	
	\begin{lemma}
		\label{lem:matrixpowers}
		Let \(M \in \ZZ^{n \times n}\), \(k \geq 1\) and \(l \geq 0\). If \(M \equiv I_n \pmod {2^k}\), then \(M^{2^{l}} \equiv I_n \pmod {2^{k+l}}\).
	\end{lemma}
	\begin{proof}
		We prove this by induction on \(l\). If \(l = 0\) the statement is trivially true. Now suppose that \(M^{2^{l}} \equiv I_n \pmod {2^{k+l}}\), i.e.\@ \(M^{2^{l}} = I_n + 2^{k+l} X\) for some integer matrix \(X\).
		Then
		\[
		M^{2^{l+1}} = (I_n + 2^{k+l} X)^2 = I_n + 2^{k+l+1} X + 2^{2(k+l)} X^2,
		\]
		so \(M^{2^{l+1}} \equiv I_n \pmod {2^{k+l+1}}\).
	\end{proof}
	
	\begin{lemma}
		\label{lem:orbitsfi}
		Let \(G\) be a group acting on a set \(X\) such that \(\card{G\backslash X} = s < \infty\). If \(H\) is a subgroup of \(G\) with index \(r < \infty\), then \(\card{H\backslash X} \leq rs\).
	\end{lemma}
	\begin{proof}
		Let \(g_1, \ldots, g_r\) be (right) coset representatives of \(H \backslash G\). For each \(x \in X\),
		\[
		G \cdot x = \set{ g \cdot x }{g \in G} = \bigcup_{i=1}^{r} \set{ h g_i  \cdot x }{h \in H} = \bigcup_{i=1}^{r} H \cdot (g_i \cdot x).
		\]
		So each \(G\)-orbit is the union of \(r\) (not necessarily distinct) \(H\)-orbits.
	\end{proof}
	
	We now combine the two preceding \zcref[noref,nocap]{lem:matrixpowers,lem:orbitsfi} to show that, for affine actions by abelian groups on \(\ZZ^n\), having finitely many orbits puts a bound on the Hirsch length of the acting group.
	
	\begin{lemma}
		\label{lem:affznorbits}
		Let \(A\) be a finitely generated abelian subgroup of \(\Aff(\ZZ^n) = \ZZ^n \rtimes \GL(n,\ZZ)\). If the number of orbits of \(\ZZ^n\) under the action of \(A\) is finite, then \(h(A) \geq n\).
	\end{lemma}
	\begin{proof}
		Since \(A\) is finitely generated, it contains a finite-index torsion-free subgroup. Since this subgroup has the same Hirsch length as \(A\), and its action on \(\ZZ^n\) still has finitely many orbits (due to \zcref{lem:orbitsfi}), we can just assume \(A\) itself is torsion-free.
		
		We can embed \(A \leq \Aff(\ZZ^n)\) into \(\GL(n+1,\ZZ)\) in the usual way:
		\[ (v,M) \mapsto \begin{pmatrix}M & v \\ 0 & 1\end{pmatrix}.\]
		Now consider the projection \(\GL(n+1,\ZZ) \to \GL(n+1,\ZZ/2\ZZ)\). Let \(C\) be the intersection of its kernel with \(A\), which has finite index in \(A\), and thus its action on \(\ZZ^n\) still has finitely many orbits (\zcref{lem:orbitsfi}).
		
		For an arbitrary integer \(k \geq 1\), consider the projection \(\ZZ \to \ZZ/2^k \ZZ\). Let \(C_k\) be the image of \(C\) in \(\GL(n+1,\ZZ/2^k\ZZ)\). For any matrix \(M \in C\), note that \(M \equiv I_{n+1} \pmod 2\). Therefore by \zcref{lem:matrixpowers} \(M^{2^{k-1}} \equiv I_{n+1}\pmod {2^k}\). Let \(r \coloneq h(A)\). Then \(C\) has a generating set of \(r\) elements and hence so does \(C_k\). But every element of \(C_k\) has order at most \(2^{k-1}\), so \(|C_k| \leq 2^{(k-1)r}\). If the action of \(C\) on \(\ZZ^n\) has \(s\) orbits, then the action of \(C_k\) on \((\ZZ/2^k\ZZ)^n\) has at most \(s\) orbits as well. Since every orbit can have at most \(|C_k|\) elements, we obtain
		\[ 2^{kn} = |(\ZZ/2^k\ZZ)^n| \leq s|C_k| \leq s 2^{(k-1)r},\]
		which implies that \(\log_2 s \geq k(n-r) + r\). Since \(s\) is fixed, this must hold for arbitrary \(k \in \NN\), which is only possible if \(n \leq r\). 
	\end{proof}
	
	\begin{proposition}
		\label{prop:abelianderv}
		Let \(A\) and \(B\) be finitely generated abelian groups and let \(\delta\colon A \to B\) be a group derivation. Then
		\[ h(A) \leq h(\ker \delta) + h(B).\]
		Equality is attained when the affine action via \(\delta\) has finitely many orbits.
	\end{proposition}
	\begin{proof}
		Let \(T\) be the torsion subgroup of \(B\), which is characteristic, hence \(A\)-invariant. Let \(p \colon B \to B/T\) be the natural projection. The quotient \(B/T\) is an \(A\)-module, is torsion-free as a group and has Hirsch length \(h(B/T)=h(B)\). The map \(\delta_T \coloneq p \circ \delta\) is then a derivation and \(\delta(\ker \delta_T)\subseteq T\), so it is finite. Moreover, for \(a_1,a_2\in \ker \delta_T\),
		\[
		\delta(a_1)=\delta(a_2) \implies \delta(a_1^{-1}a_2)=1
		\]
		by \zcref{prop:derivprops}(3). Thus \(\delta\) induces an injection from the cosets \(\ker \delta_T/\ker\delta\) to the finite set \(T\). Hence \(\ker\delta\) has finite index in \(\ker \delta_T\), and therefore \(h(\ker \delta_T) =h(\ker\delta)\). Furthermore, we have that
		\[
		p(a \ast_\delta b) = p( \delta(a)(a \cdot b) ) = \delta_T(a)(a \cdot p(b)) = a \ast_{\delta_T} p(b),
		\]
		which shows that \(p\) induces a surjection \(\aff{A}{\delta}{B} \to \aff{A}{\delta_T}{(B/T)}\). So if the affine action via \(\delta\) has finitely many orbits, so does the affine action via \(\delta_T\). Therefore, in the remainder of this proof, we may assume that \(B\) is torsion-free, i.e.\@ \(B \cong \ZZ^n\) for some \(n \in \NN\).
		
		First, we prove the inequality \(h(A) \leq h(\ker \delta) + h(B)\). We will assume for now that \(B\) has no non-trivial proper \(A\)-invariant pure subgroup. Set \(F \coloneq C_B(A) = \set{b \in B}{a \cdot b = b \quad\forall a \in A}\). This is clearly \(A\)-invariant, and moreover it is pure: suppose that \(b \in B\) and \(m \geq 1\) are such that \(b^m \in F\). Then for any \(a \in A\) we have \(a \cdot b^m = (a \cdot b)^m = b^m\). But then \(\left((a \cdot b) b^{-1}\right)^m = 1\). As \(B\) is torsion-free we conclude that \(a \cdot b = b\), i.e.\@ \(b \in F\).
		
		By our assumption, there are two possibilities: either \(F = B\) or \(F\) is trivial. If \(F = B\), then \(A\) acts trivially on \(B\) and therefore \(\delta\) is a group homomorphism. The result then follows from the first isomorphism theorem. If \(F = 1\), set \(C \coloneq \ker \rho\) where \(\rho \colon A \to \Aut(B)\). Then for all \(a \in A\) and \(c \in C\):
		\begin{align*}
			\delta(ca) &= \delta(c)(c \cdot \delta(a)) = \delta(c)\delta(a),\\
			\delta(ac) &= \delta(a)(a \cdot \delta(c)).
		\end{align*}
		Since \(A\) is abelian, \(\delta(ca) = \delta(ac)\) and hence \(a \cdot \delta(c) = \delta(c)\). Thus, \(\delta(c) \in F\), but \(F\) is trivial, so \(C \leq \ker \delta\). Then \(B\) is rationally irreducible as an \(\im \rho\)-module, so by the first isomorphism theorem and \zcref{thm:irredhirsch},
		\[ h(A) = h(\ker \rho) + h(\im \rho) \leq h(\ker \delta) + (n-1) < h(\ker \delta) + h(B).\]
		Next, we consider the general case where \(B\) may have a non-trivial proper \(A\)-invariant pure subgroup \(P\) and prove this by induction on the Hirsch length of \(B\). 
		Note that \(h(P) < h(B)\), because a pure subgroup of finite index must be the group itself. If \(h(B) \leq 1\), then every pure subgroup is either trivial or \(B\) itself, so no such \(P\) exists. So assume that \(h(B) \geq 2\) and \(B\) has a non-trivial proper \(A\)-invariant pure subgroup \(P\).
		
		Then \(\bar{\delta} \colon A \to B/P\) is again a derivation and \(h(B/P) < h(B)\). Set \(C \coloneq \ker \bar{\delta}\). By induction,
		\begin{equation}
			\label{eqn:abder1}
			h(A) \leq h(C) + h(B/P).
		\end{equation}
		Moreover, the restriction \(\delta' \colon C \to P \colon c \mapsto \delta(c)\) is also a derivation and \(h(P) < h(B)\), so again by induction
		\begin{equation}
			\label{eqn:abder2}
			h(C) \leq h(\ker \delta') + h(P).
		\end{equation}
		Since \(\ker \delta \leq C\), we have \(\ker \delta' = \ker \delta\). Thus, combining \zcref[noname]{eqn:abder1,eqn:abder2}, we get
		\[ h(A) \leq h(\ker \delta') + h(P) + h(B/P) = h(\ker \delta) + h(B). \]
		
		Second, we prove the equality when the affine action of \(A\) on \(B\) via \(\delta\) has finitely many orbits.
		Set \(D \coloneq \langle \delta(A) \rangle\). For all \(a,a' \in A\) we have
		\[
		a \cdot \delta(a')= \delta(a)^{-1} \delta(aa'),
		\]
		which shows that \(D\) is \(A\)-invariant. By \cite[Lem.~4]{howi11-a} \(B/D\) must be finite and thus \(h(B) = h(D)\). Set \(C \coloneq \ker \delta\). Note that the action of \(C\) on \(D\) is trivial. Indeed,
		for all \(a \in A\) and \(c \in C\):
		\begin{align*}
			\delta(ca) &= \delta(c)(c \cdot \delta(a)) = c \cdot \delta(a),\\
			\delta(ac) &= \delta(a)(a \cdot \delta(c)) = \delta(a),
		\end{align*}
		so \(c \cdot \delta(a) = \delta(a)\). Since \(D\) has finite index in \(B\), this implies that \(c\) acts trivially on \(B\): for every \(b \in B\), there exists an \(m \geq 1\) such that \(b^m \in D\), and then
		\[
		((c \cdot b)b^{-1})^m=(c \cdot (b^m))b^{-m}= b^m b^{-m} = 1.
		\]
		As \(B\) is torsion-free, this implies that \(c \cdot b = b\). Thus \(C\) lies in the kernel of the homomorphism \(\sigma \colon A \to \Aff(B)\) associated to the affine action. Conversely, if \(\sigma(a)\) is the identity map on \(B\), then its evaluation at \(1_B\) gives that \(\delta(a) = 1_B\). Hence \(\ker\sigma=C\).
		The group \(\im \sigma\) is a finitely generated abelian subgroup of \(\Aff(B)\), and its action on \(B\) has finitely many orbits. Hence \zcref{lem:affznorbits} implies \(h(\im \sigma) \geq h(B)\). Consequently,
		\[
		h(A/C) = h(A/\ker \sigma) = h(\im \sigma) \geq h(B),
		\]
		and therefore
		\[ h(A) = h(C) + h(A/C) \geq h(\ker \delta) + h(B). \]
		Together with the inequality proved above, this shows equality.
	\end{proof}
	
	\section{Default setup}
	\label{sec:setup}
	
	We construct a reduction used in the following two sections. Let \(G\) be a group with subgroups \(H,K\) and normal subgroup \(N\). Let \(p\) be the projection \(G \to G/N\); we will use a bar to denote the image of a subgroup or element under \(p\), e.g.\@ \(\bar{H} \coloneq p(H)\). We will also use a subscript \(N\) to denote the intersection with \(N\), e.g.\@ \(H_N \coloneq H \cap N\). Let \(\dcs{H}{G}{K}\) denote the set of all \((H,K)\)-double cosets in \(G\). Then we have the following exact sequence of pointed sets:
	
	\begin{equation}
		\label{eqn:exactseq}
		\begin{tikzcd}[row sep=large]
			1 \arrow[r] & H_N \cap K_N \arrow[r] & H \cap K  \arrow[r] & \bar{H} \cap \bar{K}
			\arrow[dll,out=0, in=180, looseness=2,overlay,swap,"\delta"] & \\
			& \dcs{H_N}{N}{K_N} \arrow[r]   &  \dcs{H}{G}{K} \arrow[r] & \dcs{\bar{H}}{\bar{G}}{\bar{K}} \arrow[r] & 1
		\end{tikzcd}
	\end{equation}
	We have omitted the base points here, as they are the natural choices. For the first three non-trivial terms, the base points are the identity elements. For the second three, the base points are the double cosets whose representative is the identity element (e.g.\@ \(HK\) for \(\dcs{H}{G}{K}\)). Exactness in a sequence of pointed sets means that \((X,x) \xrightarrow{f} (Y,y) \xrightarrow{g} (Z,z)\) is exact when \(\im(f) = g^{-1}(z)\).
	
	All of the maps in this sequence are also natural, except the connecting map \(\delta\). If \(\bar{g} \in \bar{H} \cap \bar{K}\), then there exists an \(h_{\bar{g}} \in H\) such that \(p(h_{\bar{g}}) = \bar{g}\), and similarly, there exists a \(k_{\bar{g}} \in K\) such that \(p(k_{\bar{g}}) = \bar{g}\). We can then define \(\delta\) as
	\[
	\delta\colon \bar{H} \cap \bar{K} \to \dcs{H_N}{N}{K_N} \colon \bar{g} \mapsto H_N (h_{\bar{g}} k_{\bar{g}}^{-1}) K_N.
	\]
	
	This exact sequence and its connecting map are essentially ``translated'' from the language of twisted conjugacy, where they are a frequently used tool; see e.g.\@ \cite[Thm.~2.4]{dt21-a} and \cite[Prop.~2.3]{hly23-a}.
	
	From now on, we assume that \(G\) is virtually polycyclic and \(N\) is abelian. To keep the notation short and readable, we introduce some naming conventions: we define \(I \coloneq H \cap K\) and \(J \coloneq HN \cap KN\), so that \(\bar{I} = \overline{H \cap K}\) and \(\bar{J} = \bar{H} \cap \bar{K}\). Moreover, \(D \coloneq H_NK_N\) is a subgroup of \(N\) and the quotient \(M \coloneq N/D\) coincides with \(\dcs{H_N}{N}{K_N}\). In terms of this newly introduced notation, \(\delta\) becomes
	\[ \delta\colon \bar{J} \to M \colon \bar{\jmath} \mapsto h_{\bar{\jmath}} k_{\bar{\jmath}}^{-1}D. \]
	By choosing \(N\) to be abelian, we can put additional structure on both \(M\) and \(\delta\).
	
	\begin{proposition}
		The abelian group \(M\) is a \(\bar{J}\)-module via the action
		\[ \bar{\jmath} \cdot (nD) \coloneq jnj^{-1}D, \]
		with \(j \in p^{-1}(\bar{\jmath})\), and the map \(\delta\) is a group derivation with respect to this action.
	\end{proposition}
	\begin{proof}
		First, we verify that \(M\) is a \(\bar{J}\)-module. Let \(\bar{\jmath} \in \bar{J}\) and \(j \in p^{-1}(\bar{\jmath})\). Choose \(h\in H\) and \(k\in K\) such that \(p(h) = p(k) = \bar{\jmath}\). Then there exist \(n_H, n_K \in N\) such that \(j = hn_H = kn_K\). Since \(N\) is abelian, conjugating an element of \(N\) by \(j\) is the same as conjugating it by \(h\) or by \(k\). Moreover, since \(N\) is normal in \(G\), the subgroups \(H_N\) and \(K_N\) are normal in \(H\) and \(K\), respectively. Hence
		\[
		jH_Nj^{-1} = H_N, \qquad jK_Nj^{-1} = K_N,
		\]
		and therefore \(jDj^{-1} = D\). Thus conjugation by \(j\) induces an automorphism of \(M\). This automorphism does not depend on the choice of \(j\). Indeed, if \(j'\in p^{-1}(\bar{\jmath})\), then \(j' = jn\) for some \(n \in N\), and therefore
		\[
		j'n'(j')^{-1}=jnn'n^{-1}j^{-1}=jn'j^{-1}
		\]
		for every \(n' \in N\), where we again used that \(N\) is abelian. If \(x, y \in G\) are preimages of \(\bar{x}, \bar{y}\in\bar{J}\), then \(xy\) is a preimage of \(\bar{x}\bar{y}\), and
		\[
		\bar{x}\bar{y} \cdot (nD)
		= xyn(xy)^{-1}D
		= x(yny^{-1})x^{-1}D
		= \bar{x} \cdot (\bar{y}\cdot(nD)).
		\]
		Since the identity acts trivially, this defines an action of \(\bar{J}\) on the abelian group \(M\) via automorphisms, and the latter is indeed a \(\bar{J}\)-module.
		
		Second, we verify that \(\delta\) is a derivation. For \(\bar{x},\bar{y} \in \bar{J}\), we may choose \(h_{\bar{x}},h_{\bar{y}} \in H\) and \(k_{\bar{x}}, k_{\bar{y}} \in K\) as preimages, respectively, such that \(h_{\bar{x}}h_{\bar{y}}\) and \(k_{\bar{x}}k_{\bar{y}}\) are preimages of \(\bar{x}\bar{y}\). Then
		\begin{align*}
			\delta(\bar{x}\bar{y})
			&= h_{\bar{x}}h_{\bar{y}}k_{\bar{y}}^{-1}k_{\bar{x}}^{-1}D \\
			&=(h_{\bar{x}}k_{\bar{x}}^{-1})
			\left(k_{\bar{x}}(h_{\bar{y}}k_{\bar{y}}^{-1})k_{\bar{x}}^{-1}\right)D\\
			&= \delta(\bar{x})(\bar{x} \cdot \delta(\bar{y})),
		\end{align*}
		which concludes the proof.
	\end{proof}
	
	With this knowledge of the structures of \(M\) and \(\delta\), we can construct a new virtually polycyclic group that is easier to work with, but still encodes in its subgroups the Hirsch lengths we are interested in.
	
	\begin{proposition}
		The semi-direct product \(E \coloneq M \rtimes \bar{J}\) contains subgroups \(P\) and \(Q\) such that \(h(P) = h(Q) = h(\bar{J})\) and \(h(P \cap Q) = h(\bar{I})\).
	\end{proposition}
	\begin{proof}
		Let \(\delta\) be as before. The exactness of \zcref[noname]{eqn:exactseq} tells us that \(\ker \delta = \bar{I}\). Define the subgroups \(P\) and \(Q\) as
		\begin{align*}
			P &\coloneq \set{(\delta(x),x)}{x \in \bar{J}},\\
			Q &\coloneq \set{(1,x)}{x \in \bar{J}}.
		\end{align*}
		Then \(P \cap Q \cong \ker \delta = \bar{I}\) and \(h(P) = h(Q) = h(\bar{J})\).
	\end{proof}
	
	Finally, in addition to the Hirsch lengths, this semi-direct product and its subgroups also retain the necessary information on the finiteness of the number of double cosets.
	
	\begin{proposition}
		\label{prop:equivalence}
		With \(E\), \(P\), \(Q\) and \(\delta\) defined as above, there are bijections between the \((H,K)\)-double cosets of \(G\) that meet \(N\), \(\aff{\bar{J}}{\delta}{M}\), and \(\dcs{P}{E}{Q}\).
	\end{proposition}
	\begin{proof}
		Let \(n_1, n_2 \in N\) and set \(m_i \coloneq n_iD\). The group \(\bar{J}\) acts affinely on \(M\) by
		\[ \bar{\jmath} \ast m \coloneq \delta(\bar{\jmath})(\bar{\jmath} \cdot m). \]
		First, we show that
		\[ Hn_1K = Hn_2K \iff \bar{J} \ast m_1 = \bar{J} \ast m_2.\]
		Suppose that \(n_1 = hn_2k\) for \(h \in H\) and \(k \in K\). Setting
		\[ \bar{\jmath} \coloneq p(h) = p(k^{-1}) \in \bar{H} \cap \bar{K} = \bar{J}\]
		we find that \(\delta(\bar{\jmath}) = hkD\). Thus,
		\[
		\bar{\jmath} \ast m_2
		= (\bar{\jmath} \cdot n_2D)\delta(\bar{\jmath})
		= (hn_2h^{-1}D)(hkD)
		= hn_2kD
		= m_1.
		\]
		Conversely, suppose that \(m_1= \bar{\jmath} \ast m_2\) for some \(\bar{\jmath} \in \bar{J}\). Then
		\(\bar{\jmath} \ast (n_2D) = h_{\bar{\jmath}} n_2 k_{\bar{\jmath}}^{-1}D\),
		hence \(n_1 = h_{\bar{\jmath}} n_2 k_{\bar{\jmath}}^{-1}d\) for some \(d \in D\). Let \(d = h_dk_d\) with \(h_d \in H_N\), \(k_d \in K_N\). Using that \(N\) is abelian we find
		\[n_1 = h_{\bar{\jmath}} n_2 k_{\bar{\jmath}}^{-1}h_dk_d = (h_dh_{\bar{\jmath}}) n_2 (k_{\bar{\jmath}}^{-1}k_d),\]
		so indeed \(Hn_1K = Hn_2K\).
		
		Second, identifying \(M\) with its image under the natural inclusion in \(E\), we show that
		\[ Pm_1Q = Pm_2Q \iff \bar{J} \ast m_1 = \bar{J} \ast m_2.\]
		Suppose that \(m_1 = pm_2q\) with \(p \in P, q \in Q\), i.e.
		\begin{align*}
			(m_1,\bar{1})
			= (\delta(\bar{x}),\bar{x})(m_2,\bar{1})(1,\bar{y})
			= (\delta(\bar{x})(\bar{x} \cdot m_2),\bar{x}\bar{y})
			= (\bar{x} \ast m_2,\bar{x}\bar{y})
		\end{align*}
		for certain \(\bar{x}, \bar{y} \in \bar{J}\). Then clearly we must have \(\bar{x} = \bar{y}^{-1}\) and hence \(m_1 = \bar{x} \ast m_2\). Conversely, if \(m_1 = \bar{\jmath} \ast m_2\), then \(m_1 = pm_2q\) with \(p \coloneq (\delta(\bar{\jmath}), \bar{\jmath})\) and \(q \coloneq (1,{\bar{\jmath}}^{\,-1})\). The result now follows by remarking that every \((P,Q)\)-double coset has a representative of the form \((m,\bar{1})\). Indeed, if \((m,\bar{\jmath})\) is an element of a particular \((P,Q)\)-double coset, then so is \((m,\bar{1}) = (m,\bar{\jmath})(1,\bar{\jmath}^{-1})\) since \((1,\bar{\jmath}^{-1}) \in Q\).
	\end{proof}
	
	\section{Metabelian groups}
	\label{sec:basecase}
	
	Proving \zcref{thm:mainresultA} for finitely generated abelian groups is nearly trivial, since the result then follows easily from the second isomorphism theorem. As an intermediate step towards proving it in full generality, we look at polycyclic metabelian groups.
	
	\begin{lemma}
		\label{lem:metabeliancase}
		Let \(H, K\) be subgroups of a polycyclic metabelian group \(G\). Then
		\[ h(H) + h(K) \leq h(H \cap K) + h(G).\]
		Moreover, we have equality when the number of \((H,K)\)-double cosets is finite.
	\end{lemma}
	\begin{proof}
		Set \(N \coloneq G'\) and consider the usual setup from \zcref{sec:setup}. By \zcref{prop:abelianderv}, we know that
		\begin{equation}
			\label{eqn:metabelian1}
			h(\bar{J}) \leq h(\ker \delta) + h(M) = h(\bar{I}) + h(M),
		\end{equation}
		and since both \(N\) and \(G/N\) are abelian we also have
		\begin{equation}
			\label{eqn:metabelian2}
			h(M) = h(N) - h(D) = h(N) - h(H_N) - h(K_N) + h(I_N)
		\end{equation}
		and
		\begin{equation}
			\label{eqn:metabelian3}
			h(\bar{H}) + h(\bar{K}) = h(\bar{J}) + h(\bar{H}\bar{K}) \leq h(\bar{J}) + h(\bar{G}).
		\end{equation}
		Combining \zcref[nocap]{eqn:metabelian1,eqn:metabelian2,eqn:metabelian3} gives:
		\begin{align}
			\label{eqn:metabelian4}
			\begin{split}
				h(H) + h(K) &= h(H_N) + h(K_N) + h(\bar{H}) + h(\bar{K})\\
				&\leq h(N) + h(I_N) - h(M) + h(\bar{J}) + h(\bar{G})\\
				&\leq h(I_N) + h(\bar{I}) + h(N) + h(\bar{G})\\
				&= h(I) + h(G),
			\end{split}
		\end{align}
		which proves the inequality.
		
		Now suppose that \(G\) is the union of finitely many \((H,K)\)-double cosets. Projecting to \(\bar{G}\), we have that \(\bar{H}\bar{K}\) has finite index in \(\bar{G}\) and hence inequality \zcref[noname]{eqn:metabelian3} is now an equality. By \zcref{prop:equivalence}, the finiteness of \(\card{\dcs{H}{G}{K}}\) implies that \(\aff{\bar{J}}{\delta}{M}\) is also finite, and hence inequality \zcref[noname]{eqn:metabelian1} becomes an equality as well.
		Therefore, all inequalities in \zcref[noname]{eqn:metabelian4} are now equalities.
	\end{proof}

	\section{General case}
	\label{sec:maincase}
	
	With the results from the previous two sections, we now have nearly all the necessary tools at our disposal to prove the main result of this paper. We just need one final \zcref[nocap,noref]{lem:reductionnilpbyab}.
	
	\begin{lemma}
		\label{lem:reductionnilpbyab}
		Let \(G\) be a virtually polycyclic group with subgroups \(H, K\). Let \(G_0\) be a finite-index subgroup of \(G\), and set \(X_0 \coloneq X \cap G_0\) for any subgroup \(X\) of \(G\). Then:
		\begin{itemize}
			\item \(h(X) = h(X_0)\);
			\item if \(\card{\dcs{H}{G}{K}} < \infty\), then also \(\card{\dcs{H_0}{G_0}{K_0}} < \infty\).
		\end{itemize}
	\end{lemma}
	\begin{proof}
		The first claim follows from \(\ind{X}{X_0} \leq \ind{G}{G_0} < \infty\) and \zcref{prop:hirschprops}(1). For the second claim, let \(H\times K\) act on \(G\) by \((h,k) \cdot g \coloneq hgk^{-1}\). The orbits are exactly the \((H,K)\)-double cosets, and \(G_0\) is invariant under the restriction of this action to the subgroup \(H_0 \times K_0\). As this last group has index \(\ind{H}{H_0} \ind{K}{K_0}\) in \(H \times K\), we can apply \zcref{lem:orbitsfi} to find that
		\[ \card{\dcs{H_0}{G_0}{K_0}} \leq \card{\dcs{H_0}{G}{K_0}} \leq \ind{H}{H_0}\ind{K}{K_0}\card{\dcs{H}{G}{K}} < \infty.\qedhere \]
	\end{proof}
	
	\getkeytheorem{mainA}
	\begin{proof}
		We start by proving the inequality. A result by Mal'cev \cite[Thm.~4]{malc51-b} states that polycyclic groups are nilpotent-by-abelian-by-finite, so the same holds naturally for virtually polycyclic groups. Thus, using \zcref{lem:reductionnilpbyab}, we may assume that \(G\) is nilpotent-by-abelian. We now proceed by (strong) induction on the nilpotency class of \(G'\). If \(G'\) is trivial then \(G\) is abelian, hence all subgroups are normal and everything easily follows from the second isomorphism theorem. If \(G'\) is abelian, then \(G\) is metabelian, and this was proved in \zcref{lem:metabeliancase}.
		So assume the nilpotency class \(c\) of \(G'\) is greater than \(1\) and set \(N \coloneq Z(G')\). Consider the usual setup from \zcref{sec:setup}.
		
		Set \(L \coloneq J \cap G'\). Then \(\bar{L} \leq \bar{J}\) is normal in \(E\) and commutes with \(M\). So the product \(M\bar{L}\) is normal in \(E\). But
		\[ \frac{E}{M\bar{L}} \cong \frac{\bar{J}}{\bar{L}} \cong \frac{J}{L} \cong \frac{JG'}{G'} \leq \frac{G}{G'},\]
		so the quotient \(E/M\bar{L}\) is abelian. Moreover, \(\bar{L} \leq \bar{G}'\) has nilpotency class at most \(c-1\), hence so does \(M\bar{L} \cong M \times \bar{L}\), and finally so does \(E' \leq M\bar{L}\). By the induction hypothesis,
		\begin{align}
			\label{eqn:JJIMJ}
			\begin{split}
				h(\bar{J}) + h(\bar{J})
				&= h(P) + h(Q) \\
				&\leq h(P \cap Q) + h(E)\\
				&= h(\bar{I}) + h(M) + h(\bar{J}),
			\end{split}
		\end{align}
		so \(h(\bar{J}) \leq h(\bar{I}) + h(M)\). The commutator subgroup \(\bar{G}'\) of the quotient \(\bar{G}\) has nilpotency class \(c-1\), so again the induction hypothesis implies that
		\[
		h(\bar{H}) + h(\bar{K}) \leq h(\bar{J}) + h(\bar{G}).
		\]
		\zcref[cap]{eqn:metabelian2} still holds, so we can just repeat the steps in \zcref{eqn:metabelian4} to obtain the desired inequality.
		
		Now suppose the number of \((H,K)\)-double cosets is finite. By \zcref{lem:reductionnilpbyab}, we can again assume that \(G\) is nilpotent-by-abelian and then proceed by (strong) induction on the nilpotency class \(c\) of \(G'\). The cases \(c=0\) and \(c=1\) have already been taken care of, so assume that \(c>1\) and set \(N \coloneq Z(G')\). As always, consider the usual setup. The projection \(p \colon G \to \bar{G}\) induces a surjection
		\[
		\dcs{H}{G}{K} \to \dcs{\bar{H}}{\bar{G}}{\bar{K}} \colon HgK \mapsto \bar{H}\bar{g}\bar{K},
		\]
		hence \(\card{\dcs{\bar{H}}{\bar{G}}{\bar{K}}}\) is finite. Since \(\bar{G}'\) has nilpotency class \(c-1\), the induction hypothesis gives
		\[
		h(\bar{H}) + h(\bar{K}) = h(\bar{J}) + h(\bar{G}).
		\]
		By \zcref{prop:equivalence}, there is a bijection between \(\dcs{P}{E}{Q}\) and
		\[
		\set{HgK \in \dcs{H}{G}{K}}{HgK \cap N \neq \varnothing},
		\]
		so certainly \(\card{\dcs{P}{E}{Q}}\) is finite. Moreover, the commutator subgroup \(E'\) has nilpotency class at most \(c-1\), so the induction hypothesis now provides equality in \zcref{eqn:JJIMJ}. Therefore,
		\[
		h(\bar{J}) = h(\bar{I}) + h(M).
		\]
		We can now repeat the steps from \zcref{eqn:metabelian4}, but all inequalities have become equalities.
		
		Finally, suppose that \(G\) is nilpotent and that equality holds. The set of torsion elements \(T\) then forms a finite normal subgroup of \(G\) \cite[Ch.~1, Cor.~10]{sega83-a}, so we can consider the projection \(q \colon G \to G/T\) where \(G/T\) is torsion-free, and set
		\(\tilde{G} \coloneq q(G)\), \(\tilde{H} \coloneq q(H)\), \(\tilde{K} \coloneq q(K)\), \(\tilde{I} \coloneq q(I)\) and
		\(\tilde{J} \coloneq \tilde{H} \cap \tilde{K}\).
		It is immediate that
		\[
		h(\tilde{G})=h(G),\qquad h(\tilde{H})=h(H),\qquad h(\tilde{K})=h(K),
		\]
		while showing that \(h(\tilde{J})=h(H\cap K)\) requires a bit of effort. For \(\tilde{\jmath} \in \tilde{J}\),
		choose \(h_{\tilde{\jmath}}\in H\), \(k_{\tilde{\jmath}}\in K\) such that
		\(q(h_{\tilde{\jmath}})=q(k_{\tilde{\jmath}})=\tilde{\jmath}\), and put \(t_{\tilde{\jmath}} \coloneq h_{\tilde{\jmath}}^{-1}k_{\tilde{\jmath}} \in T\). Now, for each \(t \in T \cap HK\), pick \(h_t \in H\) and \(k_t \in K\) such that \(t = h_t^{-1}k_t\). Thus, if \(t_{\tilde{\jmath}} = t\), then
		\[ h_{\tilde{\jmath}} h_t^{-1} = k_{\tilde{\jmath}} k_t^{-1} \in H \cap K\]
		and hence
		\[ \tilde{\jmath} = q(h_{\tilde{\jmath}}) = q(h_{\tilde{\jmath}} h_t^{-1} h_t) = q(h_{\tilde{\jmath}} h_t^{-1})q( h_t) \in \tilde{I} q(h_t). \]
		Thus, \(\ind{\tilde{J}}{\tilde{I}} \leq \card{T}\), so \(h(\tilde{I}) = h(\tilde{J})\). Then from the original equality it follows that
		\[
		h(\tilde{H}) + h(\tilde{K}) = h(\tilde{H} \cap \tilde{K}) + h(\tilde{G}).
		\]
		From \cite[Thm.~4.11]{hly23-a} it follows that \(\card{\dcs{\tilde{H}}{\tilde{G}}{\tilde{K}}} < \infty\). Let \(\tilde{g}_1, \ldots, \tilde{g}_r\) be representatives of the \((\tilde{H},\tilde{K})\)-double cosets in \(\tilde{G}\), and choose preimages \(g_i \in q^{-1}(\tilde{g}_i)\). Since \(T\) is normal in \(G\), it follows easily that
		\[ G = \bigcup_{i=1}^{r} \bigcup_{t \in T} Hg_i t K,\]
		and hence by the finiteness of \(T\) we conclude that \(\card{\dcs{H}{G}{K}} < \infty\).
	\end{proof}

	It may be tempting to assume that we get equality in \zcref{thm:mainresultA} when \(G = \langle H, K \rangle\). We illustrate below that this need not be the case.
	\begin{example}
		Let \(G\) be the integral Heisenberg group, given by the presentation
		\[ \langle a,b,c \mid [a,b] = c, [a,c] = [b,c] = 1\rangle.\]
		Set \(H \coloneq \langle a\rangle\) and \(K \coloneq \langle b \rangle\). Then \(G = \langle H,K\rangle\), but \[ h(H) + h(K) = 1+1 < 0 + 3 = h(H \cap K) + h(G).\]
	\end{example}
	
	We also illustrate that in general, for a non-nilpotent group \(G\) equality does not suffice to guarantee a finite number of double cosets.
	\begin{example}
		\label{exm:kleinbottle}
		Let \(G\) be the fundamental group of the Klein bottle, given by the presentation
		\[ \langle a,b \mid bab^{-1} = a^{-1}\rangle. \]
		Set \(H \coloneq \langle ab^2\rangle\) and \(K \coloneq \langle a^{-1}b^2 \rangle\). Then \(h(H) + h(K) = h(H \cap K) + h(G)\), but the double cosets \(\set{Ha^nbK}{ n \in \ZZ}\) are pairwise distinct, so \(\card{\dcs{H}{G}{K}} = \infty\).
	\end{example}
	
	\begin{corollary}
		\label{cor:derivationmainresult}
		Let \(G, H\) be virtually polycyclic groups and let \(\delta\colon G \to H\) be a group derivation. Then
		\[ h(G) \leq h(\ker \delta) + h(H),\]
		and moreover:
		\begin{itemize}
			\item if \(\aff{G}{\delta}{H}\) is finite, then the above is an equality;
			\item if \(H \rtimes G\) is nilpotent and the above is an equality, then \(\aff{G}{\delta}{H}\) is finite.
		\end{itemize}
	\end{corollary}
	\begin{proof}
		Construct the semi-direct product \(H \rtimes G\) and its subgroups
		\begin{align*}
			P &\coloneq \set{(\delta(g),g)}{g \in G},\\
			Q &\coloneq \set{(1_H,g)}{g \in G}.
		\end{align*}
		Then \(\ker \delta \cong P \cap Q\) and from \zcref{thm:mainresultA} we get
		\begin{align*}
			h(G) + h(G)
			&= h(P) + h(Q)\\
			&\leq h(P \cap Q) + h(H \rtimes G)\\
			&= h(\ker \delta) + h(H) + h(G).
		\end{align*}
		Subtracting \(h(G)\) from both sides gives the desired inequality.
		
		The statements regarding finiteness of \(\aff{G}{\delta}{H}\) follow easily from the fact that the affine orbits via \(\delta\) correspond exactly to the \((P,Q)\)-double cosets. This correspondence can be shown in exactly the same way as in the proof of \zcref{prop:equivalence}.
	\end{proof}
	
	\section{Application to twisted conjugacy}
	\label{sec:twicon}
	
	Twisted conjugation is the group action of \(G\) on \(H\) given by 
	\[G \times H \to H \colon (g,h) \mapsto \varphi(g) h \psi(g)^{-1},\]
	where \(\varphi,\psi \colon G \to H\) are group homomorphisms. The orbits are called the \((\varphi,\psi)\)-twisted conjugacy classes, and the stabiliser of \(1_H\) under this action is exactly the coincidence group (or equaliser) \(\Coin(\varphi,\psi)\). The set of all orbits is denoted by \(\R[\varphi,\psi]\), and its cardinality \(R(\varphi,\psi)\) is called the Reidemeister number of the pair \((\varphi,\psi)\).
	
	As mentioned in the introduction, the inequality from \zcref{thm:mainresultA} was stated in \cite[Prop.~3]{wong00-a}. In the given proof, it is claimed that if \(1 = G_0 \normalsub G_1 \normalsub \cdots \normalsub G_k = G\) is a subnormal series with cyclic factors of a polycyclic group \(G\), and \(H\) and \(K\) are subgroups of \(G\), then \((H \cap K \cap G_i)/(H \cap K \cap G_{i-1})\) is infinite if both \((H \cap G_i) / (H \cap G_{i-1})\) and \((K \cap G_i) / (K \cap G_{i-1})\) are. We illustrate below that this need not be the case.
	
	\begin{example}
		Let \(G \coloneq \langle a,b \mid [a,b] = 1\rangle \cong \ZZ^2\) with subgroups \(H \coloneq \langle a \rangle \cong \ZZ\) and \(K \coloneq \langle ab \rangle \cong \ZZ\). As a cyclic subnormal series, take \(G_1 \coloneq \langle b \rangle\) and \(G_2 \coloneq G\). Then
		\begin{align*}
			\frac{H \cap G_2}{H \cap G_1} \cong H \cong \ZZ, \quad \frac{K \cap G_2}{K \cap G_1} \cong K \cong \ZZ, \quad\frac{H \cap K \cap G_2}{H \cap K \cap G_1} \cong 1.
		\end{align*}
	\end{example}
	
	This proposition served as a tool to prove that, given two homomorphisms \(\varphi,\psi\) between two virtually polycyclic groups \(G\) and \(H\), there is a connection between the Hirsch lengths of \(H\) and \(G\), the coincidence group \(\Coin(\varphi,\psi)\), and the Reidemeister number \(R(\varphi,\psi)\).
	
	Unfortunately, the proof of \cite[Thm.~2]{wong00-a} relies on the aforementioned proof, and not just on the (correct) result itself. Furthermore, in \cite[Rem.~4.1]{hly23-a}, Ha, Lee and Yoo point out another gap in the proof: an orbit of the identity is erroneously assumed to always be a subgroup. This is analogous to assuming the image of a derivation is always a subgroup of its codomain; see \zcref{exm:derivimgnonsubgrp} for a counterexample.
	
	Using \zcref{thm:mainresultA}, we prove the \zcref[noref,nocap]{thm:mainresultB} below, which combines and generalises \cite[Lem.~2 and Thm.~2]{wong00-a} from torsion-free polycyclic groups to (not necessarily torsion-free) virtually polycyclic groups. For finitely generated torsion-free nilpotent groups, this result was also proved in \cite[Thm.~2.5]{gonc98-b} and \cite[Cor.~4.12]{hly23-a}.
	\getkeytheorem{mainB}
	\begin{proof}
		We follow the approach in \cite[Sec.~2]{hly23-a}. We refer to that paper for the details and limit ourselves to an outline of the proof below. Define
		\begin{align*}
			K &\coloneq \set{ (h,h) }{ h \in H },\\
			L &\coloneq \set{(\varphi(g),\psi(g))}{g \in G},\\
			N &\coloneq \ker \varphi \cap \ker \psi,
		\end{align*}
		and note that \(H \cong K\), \(G/N \cong L\) and \(\Coin(\varphi,\psi)/N \cong K \cap L\).
		Thus, applying \zcref{thm:mainresultA}, we get
		\[ h(K) + h(L) \leq h(K \cap L) + h(H \times H).\]
		Using the preceding isomorphisms, this gives that
		\[ h(H) + h(G) - h(N) \leq h(\Coin(\varphi,\psi)) - h(N) + 2h(H),\]
		and subtracting \(2h(H) - h(N)\) from both sides results in the desired inequality.
		
		Moreover, the map
		\[ \dcs{K}{(H \times H)}{L} \to \R[\varphi,\psi] \colon K(h_1,h_2)L \mapsto [h_1^{-1} h_2]_{\varphi,\psi}\]
		is a bijection, so the \((K,L)\)-double cosets correspond to the \((\varphi,\psi)\)-twisted conjugacy classes. The two statements about the finiteness of \(R(\varphi,\psi)\) then follow readily.
	\end{proof}
	
	\begin{remark}
		Above, we proved \zcref{thm:mainresultB} using \zcref{thm:mainresultA}. Conversely, if we take \zcref{thm:mainresultB} as a black-box theorem, then \zcref{thm:mainresultA} follows readily. Given subgroups \(H, K\leq G\), define the homomorphisms
		\[
		\varphi \colon H \times K \to G \colon (h,k) \mapsto h,\quad \psi \colon H \times K \to G \colon (h,k) \mapsto k.
		\]
		The coincidence group \(\Coin(\varphi,\psi)\) is isomorphic to \(H \cap K\), hence using the inequality from \zcref{thm:mainresultB} we get
		\[
		h(H\cap K) = h(\Coin(\varphi,\psi)) \geq h(H \times K) - h(G) = h(H) + h(K) - h(G).
		\]
		Moreover, the \((\varphi,\psi)\)-twisted conjugacy classes are exactly the \((H,K)\)-double cosets, so the two finiteness statements from \zcref{thm:mainresultB} directly imply those from \zcref{thm:mainresultA}.
	\end{remark}
	
	We end this article by illustrating that for non-nilpotent groups, equality does not imply that the Reidemeister number is finite.
	
	\begin{example}
		Let \(G\) be the group in \zcref{exm:kleinbottle} and let \(\varphi \colon G \to G\) be the automorphism defined by
		\[ \varphi(a) \coloneq a^{-1},\qquad \varphi(b) \coloneq b^{-1}.\]
		Then \(\Coin(\varphi,\id_G)\) is trivial, but \(R(\varphi,\id_G) = \infty\) (see \cite[Thm.~2.2]{gw09-a}).
	\end{example}
	
	\printbibliography
	
\end{document}